\documentclass[11pt]{article}
\usepackage{graphicx}
\usepackage{hhline}
\usepackage{ifxetex} 
\usepackage{setspace} 
\usepackage{amsmath,amsthm,amssymb,amsfonts}
\usepackage{amsmath}
         \usepackage{makeidx}
         \usepackage{ifxetex} 
\usepackage[all]{xy}
\usepackage[pagebackref=true,colorlinks,linkcolor=blue,citecolor=magenta]{hyperref}
\usepackage{fancybox}
\usepackage{fancyhdr}
\usepackage{tocbibind}
\usepackage{makeidx}
\makeindex
\makeatletter \oddsidemargin .4in \evensidemargin -.1in \textwidth
14.5cm \topmargin 0.1cm \textheight 22cm
\newcommand{\doublespacinggg}{\let\CS=\@currsize\renewcommand{\baselinestretch}{1.55}\tiny\CS}
\newcommand{\doublespacingg}{\let\CS=\@currsize\renewcommand{\baselinestretch}{1.75}\tiny\CS}
\doublespacinggg

\newtheorem{thm}{Theorem}[section]
 \newtheorem{cor}[thm]{Corollary}
 
   \newtheorem{examp}[thm]{Example}
  \newtheorem{defin}[thm]{Definition}
  \newtheorem{prop}[thm]{Proposition}

\newcommand{\be}{\begin{equation}}
\newcommand{\ee}{\end{equation}}
\newcommand{\bea}{\begin{eqnarray}}
\newcommand{\eea}{\end{eqnarray}}

\newcommand{\bee}{\begin{eqnarray*}}
\newcommand{\eee}{\end{eqnarray*}}

\title{On a generalization of NC-McCoy Rings}
\author{ \small Mohammad Vahdani Mehrabadi $^{*}$, 
Shervin Sahebi $^{**}$ and Hamid H. S. Javadi $^{***}$\\ $^{*,**}$ 
Department of Mathematics, Islamic Azad University,\\
Central Tehran Branch, 13185/768, Iran, \\email: 
md \underline{~}vahdani@yahoo.com; sahebi@iauctb.ac.ir 
\\ 
$^{***}$ Department of Mathematics and Computer Science, Shahed University,\\ Tehran, Iran, email: h.s.javadi@shahed.ac.ir.\\}

\begin{document}

\date{}
\maketitle \noindent \vspace{-.8cm}

\doublespacingg

\begin{center}
\begin{minipage}{11cm} \footnotesize { \textsc{Abstract:}
In the present paper we concentrate on a natural generalization of NC-McCoy rings 
that is called J-McCoy and investigate their properties. 
We prove that local rings are J-McCoy. 
Also, for an abelian ring $R$, we show that 
$R$ is J-McCoy if and only if $eR$ is J-McCoy, 
where $e$ is an idempotent element of $R$. 
Moreover, we give an example to show that the J-McCoy property 
does not pass $M_{n}(R)$, but $S(R, n), A(R, n), B(R, n)$ and $T(R, n)$ are J-McCoy.
  }

 \end{minipage}
\end{center}

 \vspace*{.4cm}

 \noindent {\footnotesize {\bf Mathematics Subject Classification 2010: 16U20, 16S36,16W20}    }

\vspace*{.2cm}

\doublespacing

\section{Introduction}
Throughout this paper, $R$ denotes an associative ring with identity. 
For notation $Nil(R)$, $M_{n}(R)$, $T_{n}(R)$, $I_{n}(R)$, $E_{ij}$, $R[x]$ 
and $N(R)$ denote the set of all nilpotent elements in $R$, 
the $n \times n$ matrix ring over $R$, 
upper triangular matrix ring over $R$, 
identity matrix over $R$, unit matrices, polynomial ring over $R$ 
and the set of all nilpotent elements in $R$, respectively. 
Rege -Chhawchharia \cite{a15} called a noncommutative ring $R$ 
right McCoy if whenever polynomials 
$f(x)=\sum_{i=0}^n a_{i}x^{i}$, 
$g(x)=\sum_{j=0}^m b_{j}x^{j}\in R[x]\setminus\{0\}$ 
satisfy $f(x)g(x)=0$, there exists nonzero elements $r\in R$ 
such that $a_{i}r=0$. 
Left McCoy rings are defined similarly. 
A number of papers have been written on McCoy property of rings 
(see, e. g., \cite{a1, a3, a9, a13, a14, a16}).
The name "McCoy" was chosen because 
McCoy \cite{a13} had noted that every commutative ring satisfies this condition. 
Victor Camilo, Tai Keun Kwak, and Yang Lee \cite{a4} called 
a ring $R$ right nilpotent coefficient McCoy(simply, right NC-McCoy) 
if whenever polynomials $f(x)=\sum_{i=0}^n a_{i}x^{i}$, 
$g(x)=\sum_{j=0}^m b_{j}x^{j}\in R[x]\setminus\{0\}$  
satisfy $f(x)g(x)=0$, there exists nonzero elements $r\in R$ 
such that $f(x)r \in N(R)[x]$. 
Left NC-McCoy rings are difined analogously, 
and a ring $R$ is called NC-McCoy if it is both left and right NC-McCoy. 
They proved for a reduced ring $R$ and $n \geq 2$, 
$M_n(R)$ is neither right nor left NC-McCoy, 
but $T_n(R)$ is a NC-McCOy ring for $n \geq 2$. 
Moreover, it is shown that  $R$ is right NC-McCoy 
if the polynomial ring $R[x]$ is right NC-McCoy 
and the converse holds if $N(R)[x] \subset N(R[x])$.

Motivated by the above results, 
we investigate a generalization of the right NC-McCoy rings 
The Jacobson radical is an important tool 
for studying the structure of noncommutative rings, 
and denoted by $J(R)$. 
A ring $R$ is said to be right J-McCoy (respectively left  J-McCoy) 
if for each pair of nonzero polynomials $f(x)=\sum_{i=0}^n a_{i}x^{i}$ and 
$g(x)=\sum_{j=0}^m b_{j}x^{j}\in R[x]\setminus\{0\}$ 
with $f(x)g(x)=0$, then there exists a nonzero element 
$r \in R$ such that $a_{i}r \in J(R)$ (respectively $rb_{j} \in J(R)$). 
A ring $R$ is called J-McCoy if it is both  left  and right J-McCoy. 
It is clear that NC-McCoy rings are J-McCoy, 
but the converse is not always true. 
If R is J-semisimple (Namely, $J(R)=0$), 
then $R$ is right J-McCoy if and only if $R$ is right McCoy. 
Moreover, for Artinian rings, the concepts of NC-McCoy and J-McCoy rings are the same.

\section{J- McCoy Rings}

 \noindent
We start this section by the following definition:
\begin{defin}\label{def 1.1}  
A ring $R$ is said to be right J-McCoy (respectively left  J-McCoy) 
if for each pair of nonzero polynomials $f(x)=\sum_{i=0}^n a_{i}x^{i}$ 
and $g(x)=\sum_{j=0}^m b_{j}x^{j}\in R[x]\setminus\{0\}$,  
$f(x)g(x)=0$ implies that there exists a nonzero element 
$r \in R$ with $a_{i}r \in J(R)$ (respectively $rb_{j} \in J(R)$). 
A ring $R$ is called J-McCoy if it is both  left  and right J-McCoy.
\end{defin}

It is clear that NC-McCoy rings are J-McCoy, 
but the converse is not always true by the following example.

\begin{examp}\label{examp1}
 Let $A$ be the 3 by 3 full matrix ring 
 over the power series ring $F[[t]]$ over a field $F$. 
 Let 
\begin{center}
$B= \{  M=(m_{ij}) \in A~|~m_{ij} \in tF[[t]]$~for~1 $\leq i,j \leq 2$ ~and 
$m_{ij}=0$~for~$i=3$ ~or~ $j=3 \}$ 
\end{center} 
and 
\begin{center}
$C= \{ M=(m_{ij}) \in A~|~m_{ij} \in F$~and~ $m_{ij}=0$~for~$i \neq j \}.$
\end{center}
Let $R$ be the subring of $A$ generated by $B$ and $C$. 
Let $F=\mathbb{Z}_{2}$. 
Note that every element of R is of the form $(a+f_{1})e_{11}+f_{2}e_{12}+f_{3}e_{21}+(a+f_{4})e_{22}+ae_{33}$ for some $a \in F$ and $f_{i} \in tF[[t]]~ (i=1, 2, 3, 4)$. 
Consider two polynomials over R,
$f(x)=te_{11}+te_{12}x+te_{21}x^{2}+te_{22}x^{3}$ and 
$g(x)= -t(e_{21}+e_{22})+t(e_{11}+e_{12})x \in R[x]$. 
Then $f(x)g(x)=0$ but there cannot exist $ 0 \neq r \in R$ 
such that $f(x)r \in N(R[x])$, concluding that $R$ is not right NC-McCoy.\\
Next we will show that $R$ is right J-McCoy. 
Let $f(x)=\sum_{i=0}^n M_{i}x^{i}$ and 
$g(x)=\sum_{j=0}^m N_{j}x^{j}$ be nonzero polynomials in $R[x]$ 
such that $f(x)g(x)=0$. 
Since $M_{i}=(a_{i}+f_{i1})e_{11}+f_{i2}e_{12}+f_{i3}e_{21}+(a_{i}+f_{i4})e_{22}+
a_{i}e_{33}$ for some $a_i \in F$ and  
$f_{ij} \in tF[[t]]~ (j=0,1,2,3,4)$,  
then for $C=te_{11}$ we have 
$M_{i}C=(a^{i}+f^{i}_{1})te_{11}+f^{i}_{3}te_{21} \in J(R)$.
Thus $R$ is right J-McCoy ring.
\end{examp}

\begin{prop}\label{proposition1}
Let $R$ be a ring and $I$ an ideal of $R$ such that 
$R/I$ is a right (resp. left) J-McCoy ring. 
If $I \subseteq J$, then $R$ is a right (resp. left) J-McCoy ring. 
\end{prop}

\begin{proof}
Suppose that $f(x)= \sum_{i=0}^m a_{i}x^{i}$ and 
$g(x)= \sum_{j=0}^n b_{j}x^{j} \in R[x]\setminus\{0\}$ such that $f(x)g(x)=0$. 
Then $(\sum_{i=0}^m \bar{a_{i}}x^{i})(\sum_{j=0}^n \bar{b_{j}}x^{j})=\bar{0}$ in $R/I$ . Thus there exists $\bar{c} \in R/I$ such that $\bar{a_{i}}\bar{c} \in J(R/I)$ 
and so $a_{i}c \in J(R)$. 
This means $R$ is right J-McCoy ring.
\end{proof}

\begin{cor}\label{Corollary}
Let $ R $ be any local ring. Then $ R $ is J-McCoy.
\end{cor}

\begin{prop}\label{proposition}
Let $R_{k}$ be a ring, where $k \in I$. 
Then $R_{k}$ is right (resp. left)  J-McCoy for each $k \in I$ 
if and only if $R=\prod_{k\in I} R_{k}$ is right (resp. left)  J-McCoy.
\end{prop}

\begin{proof}
Let each $R_{k}$ be a right J-McCoy ring and 
$f(x)= \sum_{i=0}^m a_{i}x^{i}, 
g(x)= \sum_{j=0}^n b_{j}x^{j} \in R[x]\setminus\{0\}$ 
such that $f(x)g(x)=0$, 
where $a_{i}=(a_{i}^{~(k)})$ , $b_{j}=(b_{j}^{~(k)})$. 
If there exists $t\in I$ such that $a^{(t)}_{i}=0$ 
for each $0\leq i \leq m$, 
then we have $a_{i}c=0 \in J(R)$ 
where $c=(0,0,\ldots, 1_{R_{t}},0,\ldots,0)$. 
Now suppose for each $k\in I$, 
there exists $0\leq i_{k} \leq m$ such that 
$a^{(k)}_{i_{k}}\neq 0$. 
Since $g(x)\neq 0$, there exists $t\in I$ 
and $0\leq j_{t} \leq n$ such that $ b^{(t)}_{j_{t}}\neq 0$. 
Consider $f_{t}(x)=\sum_{i=0}^m a^{(t)}_{i}x^{i}$ 
and $g_{t}(x)=\sum_{i=0}^n b^{(t)}_{j}x^{j} \in R_{t}[x]\setminus\{0\}$. 
We have $f_{t}(x)g_{t}(x)=0$. 
Thus there exists nonzero $c_{t}\in R_{t}$ 
such that $a^{(t)}_{i}c_{t}\in J(R_{t})$, 
for each $0\leq i \leq m$, 
since $R_{t}$ is right J-McCoy ring. 
Therefore, $a_{i}(0,0,\ldots, c_{t},0,\ldots,0) \in \Pi_{k \in I}J(R_{k})=J(R)$, 
for each $0\leq i \leq m$. 
Thus,  $R$ is right J-McCoy.\\
Conversely, suppose $R$ is right J-McCoy and $t\in I$. 
Let $f(x)= \sum_{i=0}^m a_{i}x^{i} , g(x)= \sum_{j=0}^n b_{j}x^{j}$ 
be nonzero polynomials in $R_{t}[x]$ 
such that $f(x)g(x)=0$. 
Set $$F(x)=\sum_{i=0}^m (0,0,\ldots,0,a_{i},0,\ldots,0)x^{i}~,
~G(x)=\sum_{j=0}^n (0,0,\ldots,0,b_{j},0,\ldots,0)x^{j} \in R[x]\setminus\{0\}.$$ 
Hence $F(x)G(x)=0$ and so there exists 
$0\neq c=(c_{i})$ such that 
$(0,0,\ldots,0,a_{i},0,\ldots,0)c\in J(R)=\Pi_{k \in I}J(R_{k})$. 
Therefore,  $a_{i}c_{t}\in J(R_{t})$ and so $R_{t}$ is right J-McCoy
\end{proof}

\begin{cor}\label{Corollary}
Let $D$ be a ring and $C$ a subring of $D$ with $1_{D}\in C$. Let 
\begin{center}
$R(C,D) = \{(d_{1}, \ldots,d_{n}, c, c, \ldots) \mid d_{i} \in D, c \in C, n\geq1\}$ 
\end{center} with addition and multiplication defined component-wise, 
$R(D,C)$ is a ring. 
Then $D$ is right (resp. left) J-McCoy if and only if $R(D,C)$ is right (resp. left) J-McCoy.
\end{cor}

\begin{thm}\label{theorem}
For a ring $R$, if $R[x]$ is  right (resp. left) J-McCoy, 
then $R$ is  right (resp. left) J-McCoy .
The converse holds if $J(R)[x] \subseteq J(R[x])$.
\end{thm} 

\begin{proof}
 Suppose  that $R[x]$  is   right J-McCoy. 
Let $f(y)=\sum_{i=0}^{n}a_{i}y^{i}$ and 
$g(y)=\sum_{j=0} ^{m}b_{j}y^{j} $ 
be nonzero plynomials  $ \in R[y] $, such that $ f(y)g(y)=0 $. 
Since $R[x]$ is  right J-McCoy and $R\subseteq R[x] $,
then there exists $0\neq c(x) = c_{0}+c_{1}x+...+c_{k}x^{k}\in R[x]$ 
such that $a_{i}c(x) \in J(R[x])$ 
and so $a_ic_i \in J(R[x]) \cap R \subseteq J(R)$ 
for all $i=0,1, \cdots ,n$. 
Since $c(x)$ is nonzero, there exists $c_l \neq 0$ 
such that $a_ic_l \in J(R)$ for $i=0,1, \cdots ,n$.\\
Conversely, suppose that $R$ is right J-McCoy and 
$f(y)g(y)=0$ for nonzero polynomials 
$f(y)=f_{0}+f_{1}y+...+f_{m}y^{m}$ and 
$g(y)=g_{0}+g_{1}y+...+g_{n}y^{n}$ in $(R[x])[y].$ 
Take the positive integer k with 
$k=\sum_{i=0}^m deg f_{i} + \sum_{j=0}^n degg_{j}$ 
where the degree of the zero polynomial is taken to be zero. 
Then $f(x^{k})$ and $g(x^{k})$ are nonzero polynomials in 
$R[x]$ and $f(x^{k})g(x^{k})=0$, 
since the set of coefficients of the $f_{i}$'s and $g_{j}$'s 
coincide with the set of coefficients of  $f(x^{k})$ and $g(x^{k})$. 
Since R is right J-McCoy, there exists a nonzero element 
$c \in R$ such that $a_{i}c \in J(R)$, 
for any coefficient $a_{i}$ of $f_{i}(x)$. 
So $f_{i}c \in J(R)[x] \subseteq J(R[x])$. 
Thus $R[x]$ is right J-McCoy.
 \end{proof}
 
\begin{prop}\label{proposition}
Let $R$ be a right (resp. left) J-McCoy ring 
and $e$ be an idempotent element of $R$. 
Then $eRe$ is a right (resp. left) J-McCoy ring. 
The converse holds $R$ is an abelian ring.
\end{prop}

\begin{proof}
Consider $f(x)= \sum_{i=0}^n ea_{i}ex^{i} ,  g(x)= \sum_{j=0}^m eb_{j}ex^{j} \in (eRe)[x]\setminus\{0\}$ 
such that $f(x)g(x)=0$. 
Since $R$ is a right J-McCoy ring, 
there exists $s \in R$ such that $(ea_{i}e) s  \in J(R)$. 
So $(ea_{i}e) ese \in eJ(R)e=J(eRe)$. 
Hence $eRe$ is right J-McCoy.
Now, assume that ~$eRe$ is a right J-McCoy ring. 
Consider ~$f(x)= \sum_{i=0}^n a_{i}x^{i}, 
g(x)= \sum_{j=0}^m b_{j}x^{j} \in R[x]\setminus\{0\}$ 
such that $f(x)g(x)=0$. 
Clearly, $ef(x)e, eg(x)e \in (eRe)[x]$ and 
$(ef(x)e)(eg(x)e)$ $=0$, since $e$ is a central idempotent element of $R$. 
Then there exists $s \in eRe$ such that 
$(ea_{i}e)s=(a_{i})s \in J(eRe)=eJ(R) \subset J(R)$. 
Hence, $R$ is right J-McCoy.
\end{proof}

Let $R$ be a ring and $\sigma$ denote an endomorphism of $R$ 
with $\sigma(1)=1$. In \cite {a6} the authors 
introduced skew triangular matrix ring as a set of 
all triangular matrices with addition point-wise and a 
new  multiplication subject to the condition 
$E_{ij}r=\sigma^{j-i}E_{ij}$. So $(a_{ij})(b_{ij})=(c_{ij})$, 
where $c_{ij}=a_{ij}b_{ij}+a_{i,i+1}\sigma(b_{i+1,j})+...+a_{ij}\sigma^{j-i}(b_{jj})$, 
for each $i \leq j$ and denoted it by $T_{n}(R,\sigma)$.
The subring of the skew triangular matrices 
with constant mail diagonal is denoted by $S(R, n, \sigma)$; 
and the subring of the skew triangular matrices 
with constant diagonals is denoted by $T(R, n, \sigma)$. 
We can denote $A=(a_{ij}) \in T(R, n, \sigma)$ by $(a_{11},...,a_{1n})$. 
Then $T(R, n, \sigma)$ is a ring 
with addition point-wise and multiplication given by $(a_{0},...,a_{n-1})(b_{0},...,b_{n-1})=(a_{0}b_{0}, a_{0}*b_{1}+a_{1}*b_{0},...,a_{0}*b_{n-1}+...+a_{n-1}*b_{0})$, 
with $a_{i}*b_{j}=a_{i}\sigma^{i}(b_{j})$, for each $i$ and $j$. 
Therefore, clearly one can see that 
$T(R, n, \sigma) \cong R[x;\sigma]/(x^{n})$ 
is the ideal generated by $x^{n}$ in $R[x;\sigma]$.
we consider the following two subrings of $S(R, n, \sigma)$, as follows (see \cite{a6}):
\begin{center}
$A(R,n,\sigma)=\sum_{j=1}^{[\frac{n}{2}]} \sum_{i=1}^{n-j+1}a_{j}E_{i,i+j-1}+\sum_{j=[\frac{n}{2}]+1}^{n} \sum_{i=1}^{n-j+1}a_{i,i+j-1}E_{i,i+j-1}$,\\
 $B(R,n,\sigma)=\{A+rE_{1k}\mid A \in A(R,n,\sigma) and r \in R \}$~~~~~ $n=2k \geq 4$
\end{center}
In the special case, when $\sigma =id_{R}$, 
we use $S(R, n), A(R, n), B(R, n)$ and $T(R, n)$ instead of $S(R, n, \sigma), A(R, n, \sigma), B(R, n, \sigma)$ and $T(R, n, \sigma)$, respectively.

\begin{prop}\label{proposition}
Let $R$ be a ring. 
Then $S$ is right J-McCoy ring, for $n \geq 2$, 
where $S$ is one of the rings 
$T_{n}(R, \sigma), S(R, n, \sigma), T(R, n, \sigma), A(R, n, \sigma)$ 
or , $B(R, n, \sigma)$.
\end{prop}

\begin{proof}
Let $f(x)=A_{0}+A_{1}x+...+A_{p} x^{p} ,  g(x)=B_{0}+B_{1}x+...+B_{q} x^{q}$ 
be elements of $S[x]$ satisfying $f(x)g(x)=0$ 
where the $(1,1)-th$ entry of $A_{i}$ is $a^{(i)}_{11}$. 
Then $A_{i}E_{1n}=a^{(i)}_{11}E_{1n} \in J(S)$ and the proof is complete.
\end{proof}

Let $ R $ and $ S $ be two rings, 
and Let $ M $ be an $ (R,S) $-bimodule. 
This means that $ M $ is a left $ R $-module and a right $ S $-module such that 
$ (rm)s=r(ms) $
for all $ r\in R $, $ m\in M $, and $ s\in S $. 
Given such a bimodule $ M $ we can form 
\begin{center}
$  T =\bigl( 
\begin{smallmatrix}
  R&M\\ 0&S
\end{smallmatrix} \bigr)=\big\lbrace \bigl( \begin{smallmatrix}r&m\\ 0&s \end{smallmatrix} \bigr) :  r\in R , m\in M , s\in S \big\rbrace$
\end{center} and define a multiplication on 
$ T  $ by using formal matrix multiplication:
\begin{center}

$ \bigl(
 \begin{smallmatrix}
  r&m\\ 0&s
\end{smallmatrix} \bigr)\bigl(
 \begin{smallmatrix}
  r^{\prime}&m^{\prime}\\ 0&s^{\prime}
\end{smallmatrix} \bigr)=\bigl(
 \begin{smallmatrix}
  rr^{\prime}&rm^{\prime}+ms^{\prime}\\ 0&ss^{\prime}
\end{smallmatrix} \bigr). $

\end{center}
This ring construction is called triangular ring $ T $. 

\begin{prop}\label{proposition}
Let $ R $ and $ S $ be two rings and $ T $ be the triangular ring 
$
T=
\bigl( \begin{smallmatrix}
  R&M\\ 0&S
\end{smallmatrix} \bigr) 
$
(where $ M $ is an $ (R,S) $-bimodule). 
Then the rings $ R $ and $ S $ are right (resp. left) J-McCoy 
if and only if $ T $ is right (resp. left) J-McCoy. 
\end{prop}
\begin{proof}
Assume that
$ R $ and
$ S $
are two right J-McCoy rings, and 
\begin{center}
$ f(x)=\bigl( \begin{smallmatrix}
  r_{0}&m_{0}\\ 0&s_{0}
\end{smallmatrix} \bigr)  +\bigl( \begin{smallmatrix}
  r_{1}&m_{1}\\ 0&s_{1}
\end{smallmatrix} \bigr)x+\cdots +\bigl( \begin{smallmatrix}
  r_{n}&m_{n}\\ 0&s_{n}
\end{smallmatrix} \bigr)x^{n}  $,
\end{center}
\begin{center}
$ g(x)=\bigl( \begin{smallmatrix}
  r^{\prime}_{0}&m^{\prime}_{0}\\ 0&s^{\prime}_{0}
\end{smallmatrix} \bigr)  +\bigl( \begin{smallmatrix}
  r^{\prime}_{1}&m^{\prime}_{1}\\ 0&s^{\prime}_{1}
\end{smallmatrix} \bigr)x+\cdots +\bigl( \begin{smallmatrix}
  r^{\prime}_{m}&m^{\prime}_{m}\\ 0&s^{\prime}_{m}
\end{smallmatrix} \bigr)x^{m}  \in T[x]$
\end{center} satisfy 
$ f(x)g(x)=0 $.
Define
\begin{center}
$ f_{r}(x) =r_{0}+r_{1}x+\cdots + r_{n}x^{n}$,
$  g_{r}(x) =r^{\prime}_{0}+r^{\prime}_{1}x+\cdots + r^{\prime}_{m}x^{m}\in R[x] $
\end{center}
 and
\begin{center}
$ f_{s}(x) =s_{0}+s_{1}x+\cdots + s_{n}x^{n}  $,
$  g_{s}(x) =s^{\prime}_{0}+s^{\prime}_{1}x+\cdots + s^{\prime}_{m}x^{m} \in S[x]. $
\end{center}
From 
$ f(x)g(x)=0 $, we have 
$ f_{r}(x)g_{r}(x)= f_{s}(x)g_{s}(x)=0 $, and since $ R $ and $ S $ are right J-McCoy rings, so there exists $c \in R$ and $d \in S$ such that $ r_{i}c\in J(R) $ and $  s_{i}d\in J(S) $ for any 
$ 1\leq i \leq n $ and $ 1\leq j \leq m $.
Now from the fact that 
$ J(T)= \bigl( \begin{smallmatrix}
  J(R)&M\\ 0&J(S)
\end{smallmatrix} \bigr) $, we obtain that 
$ \bigl( \begin{smallmatrix}
  r_{i}&m_{i}\\ 0&s_{i}
\end{smallmatrix} \bigr)\bigl( \begin{smallmatrix}
  c & 0\\ 0& d
\end{smallmatrix} \bigr) \in J(T) $ for any
$ i,j $. Hence 
$ T $
is a right J-McCoy ring.
Conversely, let $ T $ be a right J-McCoy ring,
$ f_{r}(x) =r_{0}+r_{1}x+\cdots + r_{n}x^{n}$,
$  g_{r}(x) =r^{\prime}_{0}+r^{\prime}_{1}x+\cdots + r^{\prime}_{m}x^{m}\in R[x] $, 
such that
$ f_{r}(x)g_{r}(x)=0 $, and 
$ f_{s}(x) =s_{0}+s_{1}x+\cdots + s_{n}x^{n}  $,
$  g_{s}(x) =s^{\prime}_{0}+s^{\prime}_{1}x+\cdots + s^{\prime}_{m}x^{m} \in S[x] $, 
such that
$ f_{s}(x)g_{s}(x)=0 $.
If
\begin{center}
$ f(x)=\bigl( \begin{smallmatrix}
  r_{0}&0\\ 0&s_{0}
\end{smallmatrix} \bigr)  +\bigl( \begin{smallmatrix}
  r_{1}&0\\ 0&s_{1}
\end{smallmatrix} \bigr)x+\cdots +\bigl( \begin{smallmatrix}
  r_{n}&0\\ 0&s_{n}
\end{smallmatrix} \bigr)x^{n}  $
and
$ g(x)=\bigl( \begin{smallmatrix}
  r^{\prime}_{0}&0\\ 0&s^{\prime}_{0}
\end{smallmatrix} \bigr)  +\bigl( \begin{smallmatrix}
  r^{\prime}_{1}&0\\ 0&s^{\prime}_{1}
\end{smallmatrix} \bigr)x+\cdots +\bigl( \begin{smallmatrix}
  r^{\prime}_{m}&0\\ 0&s^{\prime}_{m}
\end{smallmatrix} \bigr)x^{m}  \in T[x]$
\end{center}
Then from
$ f_{r}(x)g_{r}(x)=0  $ and $ f_{s}(x)g_{s}(x)=0  $ it folows that
$ f(x)g(x)=0 $.
Since 
$ T $ is a right J-McCoy ring so there exists $\bigl( \begin{smallmatrix}
  c &m\\ 0& d 
\end{smallmatrix} \bigr) \in T$ such that
$ \bigl( \begin{smallmatrix}
  r_{i}&0\\ 0&s_{i}
\end{smallmatrix} \bigr)\bigl( \begin{smallmatrix}
  c&m\\ 0&d
\end{smallmatrix} \bigr)\in J (T)=\bigl( \begin{smallmatrix}
  J(R)&M\\ 0&J(S)
\end{smallmatrix} \bigr) $.
Thus
$ r_{i} c\in J(R)$ and
$ s_{i} d \in J(S)$ for any $ i,j $. This shows that 
$ R $ and $ S $ are right J-McCoy rings.
\end{proof}

The following example shows that, if $R$ is right J-McCoy ring, 
then $M_{n}(R)$ is  not necessary right J-McCoy.

\begin{examp}\label{examp 1}
 Let $\mathbb{Z}$ be the set of integers. 
 It's clear that $\mathbb{Z}$ is J-McCoy, 
 but $M_{3}(\mathbb{Z})$ is not right J-McCoy. 
 For
 \begin{center}
$f(x)=\left(\begin{array}{cccc}
 1 & x &  x^{2} \\ x^{3} & x^{4} &  x^{5} \\ x^{6} & x^{7} &  x^{8}
\end{array}\right)$~and~$g(x)=\left(\begin{array}{cccc}
 x & x &  x \\ -1 & -1 &  -1 \\ 0 & 0 &  0
\end{array}\right)$
\end{center} 
in $M_{3}(\mathbb{Z})[x]$, 
we have $f(x)g(x)=0$. 
Assume to the contrary that $M_{3}(\mathbb{Z})$ is right J-McCoy, 
then there exists $c=(c_{ij}) \in M_{3}(\mathbb{Z})$ 
such that $(E_{ij}c) \in J(M_{3}(\mathbb{Z}))$ 
for $i,j=1,2,3$. 
Assume $i=1, j=2$. 
Therefore, $\left(\begin{array}{cccc}
 1 & 0 &  0 \\ -c_{11} & 1-c_{12} &  -c_{13} \\ 0 & 0 &  1
\end{array}\right) $
is invertible in $M_{3}(\mathbb{Z})$. 
Hence 
$$\left(\begin{array}{cccc}
 1 & 0 &  0 \\ c_{11}(1-c_{12})^{-1} & (1-c_{12})^{-1} & c_{13}(1-c_{12})^{-1}  \\ 0 & 0 &  1
\end{array}\right) \in M_{3}(\mathbb{Z})$$
and so $c_{12}=0$. 
  Similarly, $c_{ij}=0$ for each $i , j$. This implies $c=0$,  
  which is a contradiction.
\end{examp}

 Let $S$ denote a multiplicatively closed subset of a ring 
 $R$ consisting of central regular elements. 
 Let $RS^{-1}$ be the localization of $R$ at $S$. Then we have:
 
\begin{thm}\label{theorem 0}
For a ring $R$, if $R$ is  right (resp. left) J-McCoy, 
then $RS^{-1}$ is  right (resp. left) J-McCoy.
\end{thm}

\begin{proof}
Suppose that $R$ is right J-McCoy. 
Let $f(x)=\sum_{i=0}^n a_{i}c_{i}^{-1}x^{i}$,  
$g(x)=\sum_{j=0}^m b_{j}d_{j}^{-1}x^{j} $ 
be nonzero elements in $(RS^{-1})[x]$  
such that $f(x)g(x)=0$. 
Let $a_{i}c^{-1}_{i} = c^{-1}a'_{i}$ and 
$b_{j}d^{-1}_{j}=d^{-1}b'_{j}$ with 
$c , d$ regular elements in $R$. 
So $f'(x)g'(x)=0$ such that 
$f'(x)=\sum_{i=0}^n a'_{i}x^{i}$ and 
$g'(x)=\sum_{j=0}^m b'_{j}x^{j} \in R[x]\setminus\{0\}$. 
Since $R$ is right J-McCoy, 
there exists  $r\in R\setminus\{0\}$ 
such that $a'_{i}r \in J(R)$ for each $i$, 
equivalently we have $1-ta'_{i}r$ 
is left invertible in $R$ for each $t\in R$. 
So 
$c^{-1}w^{-1}(1-tw^{-1}a_ic_i^{-1}rcw) =
c^{-1}w^{-1}-tw^{-1}a_ic_i^{-1}r$ 
is left invertible in $RS^{-1}$,
for each $tw^{-1} \in RS^{-1}$ and so 
$a_ic_i^{-1}rcw \in J(RS^{-1})$.  
Thus $RS^{-1}$ is  right J-McCoy.  
\end{proof}

\begin{cor}\label{Corollary}
For a ring $R$ , let $R[x]$ be a right (resp. left) J-McCoy ring. 
Then $R[x,x^{-1}]$ is a right (resp. left) J-McCoy ring.
\end{cor}
\begin{proof} Let $\Delta =\{1,x,x^{2},...\}$. 
Then clearly $\Delta$ is multiplicatively closed subset of $R[x]$. 
Since $R[x,x^{-1}]=\Delta^{-1}R$, 
it follows that $R[x,x^{-1}]$ is right J-McCoy.
\end{proof}

  \singlespacing

\end{document}